\documentclass[letter,final,twosided]{amsart}
\usepackage{amsmath}
\usepackage{paralist}
\usepackage[english]{babel}
\usepackage{amsfonts}
\usepackage{amscd}
\usepackage{amssymb}
\usepackage{amsmath}
\usepackage{amstext}
\usepackage{latexsym}
\usepackage{xspace}
\usepackage{hyperref}

\newcommand{\baseRing}[1]{\ensuremath{\mathbb{#1}}}
\newcommand{\R}{\baseRing{R}}
\newcommand{\HH}{\baseRing{H}}
\newcommand{\Qo}{{\ensuremath{\hat{1}}}}
\newcommand{\Qi}{{\ensuremath{\hat{\imath}}}}
\newcommand{\Qj}{{\ensuremath{\hat{\jmath}}}}
\newcommand{\Qk}{{\ensuremath{\hat{k}}}}
\newcommand{\jgU}{\ensuremath{\mathfrak{U}}\xspace}
\newcommand{\jconj}{\overline}
\newcommand{\jdef}[1]{\index{#1}\emph{#1}}
\newcommand{\stext}[1]{\ensuremath{\quad\text{#1}\quad}}
\newcommand{\trint}{\;\;\makebox[0pt]{$\top$}\makebox[0pt]{$\cap$}\;\;}

\renewcommand{\iff}{\ensuremath{\Leftrightarrow}\xspace}

\newcommand{\CC}{\ensuremath{{\mathcal{A}}}}
\newcommand{\CCp}[1]{\ensuremath{{\mathcal{A}^{#1}}}}
\newcommand{\DC}{\ensuremath{{\mathcal{A}_d}}}
\newcommand{\DCp}[1]{\ensuremath{{\mathcal{A}_d^{#1}}}}
\newcommand{\HLds}[1]{\ensuremath{\overline{h_d^{{#1}}}}}
\newcommand{\HLd}[1]{\ensuremath{{h_d^{{#1}}}}}
\newcommand{\SG}{\ensuremath{G}\xspace}
\newcommand{\jgsg}{\ensuremath{\mathfrak{g}}\xspace}
\newcommand{\lie}[1]{\ensuremath{Lie(#1)}}

\numberwithin{equation}{section}

\providecommand{\abs}[1]{\ensuremath{\left\lvert{#1}\right\rvert}}
\providecommand{\norm}[1]{\ensuremath{\left\lVert{#1}\right\rVert}}
\providecommand{\ip}[2]{\ensuremath{\left\langle{#1},{#2}\right\rangle}}

\DeclareMathOperator{\im}{Im}

\newcommand{\ie}{\textsl{i.e.}\xspace}
\newcommand{\ti}[1]{\widetilde{#1}}
\newcommand{\conj}{\overline}

\newtheorem{theorem}{Theorem}[section]

\newtheorem{lemma}[theorem]{Lemma}
\newtheorem{proposition}[theorem]{Proposition}

\theoremstyle{definition}
\newtheorem{definition}[theorem]{Definition}
\newtheorem{remark}[theorem]{Remark}
\newtheorem{example}[theorem]{Example}

\title[{Discrete connections on principal bundles}] {A geometric
  approach to discrete connections\\on principal bundles}

\author[Javier Fern\'andez and Marcela Zuccalli]{}



 \email{jfernand@ib.edu.ar}
 \email{marcezuccalli@gmail.com}

 \thanks{This research was partially supported by grants from the
   Universidad Nacional de Cuyo, Universidad Nacional de La Plata and
   CONICET}

\begin{document}

\bibliographystyle{amsplain}

\maketitle

\centerline{\scshape Javier Fern\'andez}

\medskip {\footnotesize

  \centerline{Instituto Balseiro, Universidad Nacional de Cuyo --
    C.N.E.A.}

  \centerline{ Av. Bustillo 9500, San Carlos de
    Bariloche, R8402AGP, Rep\'ublica Argentina} }

\medskip 

\centerline{\scshape Marcela Zuccalli} \medskip
{\footnotesize
% please put the address of the second  and third author
  \centerline{Departamento de Matem\'atica, Facultad de Ciencias Exactas}
      \centerline{Universidad Nacional de La Plata}
  % \centerline{Other lines}
  \centerline{50 y 115, La Plata, Buenos Aires, 1900, Rep\'ublica
    Argentina}
}

%%%%%%%%%%%%%%%%%%%%%%%%%%%%%%%%%%%%%%%%%%%%%%%%%%%%%%%%%%%%%%%%

\bigskip

\begin{abstract}
  This work revisits, from a geometric perspective, the notion of
  discrete connection on a principal bundle, introduced by M. Leok,
  J. Marsden and A. Weinstein. It provides precise definitions of
  discrete connection, discrete connection form and discrete
  horizontal lift and studies some of their basic properties and
  relationships. An existence result for discrete connections on
  principal bundles equipped with appropriate Riemannian metrics is
  proved.
\end{abstract}

%%%%%%%%%%%%%%%%%%%%%%%%%%%%%%%%%%%%%%%%%%%%%%%%%%%%%%%%%%%%%%%%

\section{Introduction}
\label{sec:introduction}

The study of symmetries is a central part of many areas of Mathematics
and Physics. In the Differential Geometric setting, principal bundles
provide a powerful instrument to model many symmetric
systems. Connections on principal bundles are very convenient tools,
especially in the topologically nontrivial case. Among other things,
using connections, the equations of motion of mechanical systems can
be written globally and, also, connections capture the complexity of
the bundle via, for instance, the associated curvature.

Roughly speaking, mechanical systems are continuous time dynamical
systems on the tangent bundle $TQ$ of a manifold $Q$, whose dynamics
is defined using a variational principle. In the same spirit, discrete
mechanical systems are usually introduced as discrete time dynamical
systems on $Q\times Q$ whose dynamics is defined using a variational
principle~\cite{ar:marsden_west-discrete_mechanics_and_variational_integrators}.
The main motivation is that, for continuous time, one has velocities
(tangent vectors), whereas when the time is discrete, one has pairs of
points (close to one another).

Let $\SG$ be a Lie group acting on $Q$ via the (left) action
$l^Q:\SG\times Q\rightarrow Q$ in such a way that the quotient mapping
$\pi:Q\rightarrow Q/\SG$ is a principal $\SG$-bundle. The
\jdef{vertical bundle} $\mathcal{V}$ is defined, at each $q\in Q$, by
$\mathcal{V}(q):=T_q l^Q_\SG(q) =\{\xi_Q(q) : \xi \in \lie{\SG}\}
\subset T_qQ$, where $l^Q_\SG(q)$ is the $\SG$-orbit through $q$ and
$\xi_Q(q)$ is the infinitesimal generator of $l^Q$ at $q$. A
\jdef{connection} $\CC$ on the principal $\SG$-bundle $\pi$ is a
$\SG$-equivariant distribution $Hor_\CC$ on $Q$ complementing
$\mathcal{V}$. That is, at each $q\in Q$, there is a subspace
$Hor_\CC(q)\subset T_qQ$ such that, for each $v_q\in T_qQ$, there is a
unique decomposition
\begin{equation}\label{eq:VH_CC_decomposition}
  v_q = \underbrace{\xi_Q(q)}_{\in \mathcal{V}(q)} + 
  \underbrace{v_q-\xi_Q(q)}_{\in Hor_\CC(q)}.
\end{equation}
for some $\xi\in \jgsg :=\lie{\SG}$. 

Trying to extend techniques that were used to analyze the reduction of
symmetric mechanical systems
(\cite{bo:cendra_marsden_ratiu-lagrangian_reduction_by_stages}
and~\cite{ar:cendra_marsden_ratiu-geometric_mechanics_lagrangian_reduction_and_nonholonomic_systems})
to the case of discrete mechanical systems, M. Leok, J. Marsden and
A. Weinstein introduced
in~\cite{ar:leok_marsden_weinstein-a_discrete_theory_of_connections_on_principal_bundles}
and~\cite{th:leok-thesis} a notion of \jdef{discrete connection} on a
principal bundle that mimics the notion of connection. Their
definition states that a discrete connection $\DC$ is a
$\SG$-equivariant subset $Hor_\DC \subset Q\times Q$ that is
complementary to the \jdef{discrete vertical bundle}
$\mathcal{V}_d:=\{(q,l^Q_g(q))\in Q\times Q: g\in \SG\}$ where, for
$g\in\SG$, $l^Q_g:Q \rightarrow Q$ is defined by
$l^Q_g(q):=l^Q(g,q)$. Complementary means that every $(q_0,q_1)\in
Q\times Q$ can be decomposed uniquely in a vertical part and a
horizontal part in such a way that
\begin{equation}\label{eq:pair_decomposition_Ad}
  (q_0,q_1) = \underbrace{(q_0,l^Q_g(q_0))}_{\in \mathcal{V}_d} \cdot
  \underbrace{(q_0,l^Q_{g^{-1}}(q_1))}_{\in Hor_\DC}
\end{equation}
for some $g\in\SG$. The composition of vertical and arbitrary pairs
(based at the same point $q_0$) is defined by $(q_0,l^Q_g(q_0))\cdot
(q_0,q_1) := (q_0,l^Q_g(q_1))$.  The basic intuition is that tangent
vectors in $TQ$ become finite curves and, eventually, pairs of points,
elements of $Q\times Q$.  Vertical vectors based at $q\in Q$ are those
tangent vectors pointing in the direction of the group action which,
for finite time, leads to pairs of the form $(q,l^Q_g(q))$ for $g\in
\SG$. With this motivation,~\eqref{eq:pair_decomposition_Ad} is the
discrete analogue of~\eqref {eq:VH_CC_decomposition}.

Discrete connections have been successfully used to study the
reduction of discrete mechanical systems (see
\cite{ar:marrero_martin_martinez-discrete_lagrangian_and_hamiltonian_mechanics_on_lie_groupoids},~\cite{ar:fernandez_tori_zuccalli-lagrangian_reduction_of_discrete_mechanical_systems}, and~\cite{ar:fernandez_tori_zuccalli-lagrangian_reduction_of_discrete_mechanical_systems_by_stages}).

In their work, Leok, Marsden and Weinstein do not provide a thorough
definition of discrete connection, although they discuss the
equivalence between this notion and other approaches. For instance,
they relate a discrete connection to what they call a discrete
connection form and also to a discrete horizontal lift. They also give
an interpretation in terms of splittings of a certain discrete Atiyah
sequence, although the groupoid setting for this very intriguing
approach is not detailed. The purpose of the current paper is to give
a precise definition of discrete connection on a principal bundle and
analyze some of its more basic properties. Additional geometric
properties like parallel transport, holonomy and curvature will be
discussed elsewhere.

In Section~\ref{sec:geometric_definition} we define discrete
connections and study some elements associated to them: domain and
slices. We can associate other objects to a discrete connections $\DC$
that, in turn, can be used to characterize completely $\DC$. Two such
objects are the discrete connection form and the discrete horizontal
lift, that are analyzed in Sections~\ref{sec:discrete_connection_form}
and~\ref{sec:horizontal_lift} respectively. Last, in
Section~\ref{sec:existence_of_discrete_connections}, we prove an
existence result for discrete connections on principal bundles
equipped with an adequate Riemannian metric.

\vskip .2cm

\emph{Notation:} when $l^Q$ is the left $\SG$-action on $Q$,
$l^{Q\times Q}$ is the induced diagonal $\SG$-action on $Q\times
Q$. The quotient mapping from a space to its space of orbits is
denoted by $\pi$ and $p_j$ is the projection from a Cartesian product
onto its $j$-th factor. Given the maps $f_i:X_i\rightarrow Y_i$ (for
$i=1,2$), their Cartesian product is $f_1\times f_2: X_1\times X_2
\rightarrow Y_1\times Y_2$ with $(f_1\times f_2)(x_1,x_2) :=
(f_1(x_1),f_2(x_2))$.

%%%%%%%%%%%%%%%%%%%%%%%%%%%%%%%%%%%%%%%%%%%%%%%%%%%%%%%%%%%%%

\section{Geometric definition}
\label{sec:geometric_definition}

In order to use discrete connections in geometry, it is important that
they are well defined and that the objects that we associate to them
(discrete connection form, horizontal lift, etc.) be smooth.

\begin{definition}
  \label{def:discrete_connection}
  Let $Hor\subset Q\times Q$ be an $l^{Q\times Q}$-invariant
  submanifold containing the diagonal $\Delta_Q\subset Q\times Q$.  We
  say that $Hor$ defines the \jdef{discrete connection} $\DC$ on the
  principal bundle $\pi:Q\rightarrow Q/\SG$ if $(id_Q\times
  \pi)|_{Hor}:Hor\rightarrow Q\times (Q/\SG)$ is an injective local
  diffeomorphism. We denote $Hor$ by $Hor_{\DC}$.
\end{definition}

\begin{remark}
  It is possible to consider the slightly more general notion of
  \jdef{affine discrete connection} that replaces the condition
  $\Delta_Q \subset Hor$ with the requirement that $Hor$ contains the
  graph of a smooth map $\gamma:Q\rightarrow Q$. Discrete connections
  correspond to the case $\gamma:=id_Q$. This more general notion has
  been used
  in~\cite{ar:fernandez_tori_zuccalli-lagrangian_reduction_of_discrete_mechanical_systems}
  in order to construct discrete connections associated to
  nonvanishing conserved discrete momenta of discrete mechanical
  systems.
\end{remark}

The following Lemma, whose proof is straightforward, provides a
convenient way to characterize discrete connections.

\begin{lemma}\label{le:dconn_diffeo_vs_local_diffeo_and_intersection}
  The requirement that $(id_Q\times \pi)|_{Hor}:Hor \rightarrow
  Q\times (Q/\SG)$ be an injective local diffeomorphism in
  Definition~\ref{def:discrete_connection} is equivalent to the
  following two assertions being true simultaneously.
  \begin{enumerate}
  \item \label{it:dconn_diffeo_vs_local_diffeo_and_intersection-ldiff}
    $(id_Q\times \pi)|_{Hor}:Hor\rightarrow Q\times (Q/\SG)$ is a local
    diffeomorphism and
  \item \label{it:dconn_diffeo_vs_local_diffeo_and_intersection-inter}
    $l^{Q\times Q_2}_g(Hor)\cap Hor = \emptyset$ for all $g\neq e$ in
    $\SG$. Here $l^{Q\times Q_2}_g(q_0,q_1):=(q_0,l^Q_g(q_1))$.
  \end{enumerate}
\end{lemma}

Recall that a smooth map $f:X\rightarrow Y$ is \jdef{transversal} to
the submanifold $Z\subset Y$ if $\im(df(x)) + T_{f(x)}Z = T_{f(x)}Y$,
for all $x\in f^{-1}(Z)$; this situation is denoted by $f\trint
Z$. When $f\trint Z$ and $f^{-1}(Z)\neq \emptyset$, $f^{-1}(Z)\subset
X$ is a regular submanifold (see Theorem on page 28
of~\cite{bo:Guillemin-Pollack-differential_topology}).  When
$Z_1,Z_2\subset Y$ are submanifolds, they \jdef{intersect
  transversely} when $i_{Z_1}\trint Z_2$, where $i_{Z_1} :
Z_1\rightarrow Y$ is the inclusion map; this is denoted by $Z_1\trint
Z_2$. In particular, if not empty, $Z_1\cap Z_2\subset Y$ is a
submanifold when $Z_1\trint Z_2$. We refer
to~\cite{bo:Guillemin-Pollack-differential_topology} for more on
transversality. For any $q\in Q$ define the smooth map
$i_q:Q\rightarrow Q\times Q$ by $i_q(q') = (q,q')$.  The next result
proves the basic properties of a discrete connection.

\begin{proposition}\label{prop:dconn_basic_prop}
  Let $Hor_\DC$ be a discrete connection on the principal $\SG$-bundle
  $\pi:Q\rightarrow Q/\SG$. Then, the following statements are true.
  \begin{enumerate}
  \item \label{it:dconn_basic_prop-open_sets} Let $\jgU:= l^{Q\times
      Q_2}_\SG(Hor_\DC) = \{(q_0,l^Q_g(q_1))\in Q\times Q:
    (q_0,q_1)\in Hor_\DC,\, g\in\SG\} \subset Q\times Q$,
    $\jgU':=(id_Q\times \pi)(Hor_\DC) \subset Q\times (Q/\SG)$, and
    $\jgU'':=(\pi\times \pi)(Hor_\DC)\subset (Q/\SG)\times
    (Q/\SG)$. Then $\jgU$, $\jgU'$ and $\jgU''$ are open subspaces of
    the corresponding spaces.
  \item \label{it:dconn_basic_prop-GxG_action} $\jgU$ is
    $\SG\times\SG$-invariant for the product action on $Q\times
    Q$. Also, $\jgU = (\pi\times \pi)^{-1}(\jgU'')$.
  \item \label{it:dconn_basic_prop-trint_i_q} For any $q\in Q$, $i_q
    \trint Hor_\DC$. Furthermore, $Hor^2(q):= i_q^{-1}(Hor_{\DC})
    \subset Q$ is a submanifold of dimension $\dim(Q)-\dim(\SG)$.
  \item \label{it:dconn_basic_prop-trint_Hor^2(q)_V_d(q)} For any
    $q\in Q$, $Hor^2(q) \trint \mathcal{V}_d(q)$, where
    $\mathcal{V}_d(q) := l^Q_\SG(q)$. More precisely, $Hor^2(q) \cap
    \mathcal{V}_d(q) = \{q\}$ and $T_qHor^2(q) \oplus
    T_q\mathcal{V}_d(q) = T_qQ$.
  \end{enumerate}
\end{proposition}

\begin{proof}
  Being $\DC$ a discrete connection, $\jgU'$ is open and, as $\jgU:=
  l^{Q\times Q_2}_\SG(Hor_\DC) = (id_Q\times \pi)^{-1}(\jgU')$ with
  $id_Q\times \pi$ continuous, $\jgU$ is also open. Being $\pi$ a
  fiber bundle, it is an open map and so is $\pi\times \pi$;
  consequently, $(\pi\times \pi)(\jgU) =(\pi\times \pi)(Hor_\DC) =
  \jgU''$ is open, proving
  part~\ref{it:dconn_basic_prop-open_sets}. The
  $\SG\times\SG$-invariance of $\jgU$ follows from the
  $\SG$-invariance of $Hor_\DC$ by direct computation. This
  $\SG\times\SG$-invariance together with $(\pi\times \pi)(\jgU) =
  \jgU''$ lead to the proof of
  point~\ref{it:dconn_basic_prop-GxG_action}.

  Let $q'\in i_q^{-1}(Hor_\DC)$. For $(v,v')\in T_{(q,q')}(Q\times
  Q)$, as $d((id_Q\times \pi)|_{Hor_\DC})(q,q')$ is an isomorphism,
  there is a unique $(\ti{v},\ti{v'}) \in T_{(q,q')}Hor_\DC$ such that
  $d((id_Q\times \pi)|_{Hor_\DC})(q,q')(\ti{v},\ti{v'}) = d(id_Q\times
  \pi)(q,q')(v,v')$. It is easy to check that $\ti{v} = v$, so that
  $(v,v') = (v,\ti{v'}) + (0,v'-\ti{v'})$, where the first term is in
  $T_{(q,q')} Hor_\DC$ and the second in $\im(di_q(q')))$. This proves
  the transversality condition in
  point~\ref{it:dconn_basic_prop-trint_i_q}. The rest of this point is
  a consequence of the transversality condition (see Theorem on page
  28 of~\cite{bo:Guillemin-Pollack-differential_topology}) and the
  fact that $q\in i_q^{-1}(Hor_\DC)$.

  By
  condition~\ref{it:dconn_diffeo_vs_local_diffeo_and_intersection-inter}
  in Lemma~\ref{le:dconn_diffeo_vs_local_diffeo_and_intersection},
  $Hor^2(q) \cap \mathcal{V}_d(q) = \{l^Q_g(q): (q,l^Q_g(q)) \in
  Hor_\DC\} =\{q\}$. For $v'\in T_qQ$, as $d((id_Q\times
  \pi)|_{Hor_\DC})(q,q)$ is an isomorphism, there is a unique
  $(\ti{v},\ti{v'}) \in T_{(q,q)}Hor_\DC$ such that $d((id_Q\times
  \pi)|_{Hor_\DC})(q,q)(\ti{v},\ti{v'}) = d((id_Q\times
  \pi))(q,q)(0,v')$. Then, $\ti{v} = 0$ and $\ti{v'} - v' \in
  \ker(d\pi(q)) = T_q(l^Q_\SG(q))$. Hence $di_q(q)(v') = (0,v') =
  (0,\ti{v'}) + (0, v' - \ti{v'})$, with $(0,\ti{v'}) \in
  T_{(q,q)}Hor_\DC$. Consequently, $v' = \ti{v'} + (v'-\ti{v'})$ with
  $v'-\ti{v'} \in T_q\mathcal{V}_d(q) = T_q(l^Q_\SG(q))$ and $\ti{v'}
  \in T_qHor^2(q) = (di_q(q))^{-1}(T_{(q,q)}Hor_\DC)$, proving the
  transversality part of
  point~\ref{it:dconn_basic_prop-trint_Hor^2(q)_V_d(q)}. The direct
  sum property follows from the dimensions of the subspaces.
\end{proof}

Given a discrete connection $\DC$, the open subset $\jgU\subset
Q\times Q$ defined in part~\ref{it:dconn_basic_prop-open_sets} of
Proposition~\ref{prop:dconn_basic_prop} will be called the
\jdef{domain} of $\DC$. The submanifolds $Hor^2(q)\subset Q$
introduced in part~\ref{it:dconn_basic_prop-trint_i_q} of the same
result will be the \jdef{horizontal slices}.

\begin{proposition}\label{prop:CD_imp_pair_decomposition}
  Let $\DC$ be a discrete connection with domain $\jgU$ on the
  principal $\SG$-bundle $\pi:Q\rightarrow Q/\SG$. For any
  $(q_0,q_1)\in \jgU$, there is a unique $g\in\SG$ such
  that~\eqref{eq:pair_decomposition_Ad} holds.
\end{proposition}
\begin{proof}
  The existence of $g$ follows from the definitions of $\jgU$ and the
  composition $\cdot$. The uniqueness of $g$ is a consequence of
  $(id_Q\times \pi)|_{Hor_\DC}$ being a injective.
\end{proof}

%%%%%%%%%%%%%%%%%%%%%%%%%%%%%%%%%%%%%%%%%%%%%%%%%%%%%

\section{Discrete connection form}
\label{sec:discrete_connection_form}

A convenient way of describing and using a connection on a principal
$\SG$-bundle $\pi:Q\rightarrow Q/\SG$ is through the associated
\jdef{connection $1$-form}, that is a Lie algebra valued $1$-form
$\CC:TQ\rightarrow \jgsg$ such that $\CC(v_q) = \xi$, where $\xi$ is
as in~\eqref{eq:VH_CC_decomposition}
(see~\cite{bo:kobayashi_nomizu-foundations-v1} for more details). In
the same spirit, when $\DC$ is a discrete connection on the same
bundle, the element $g\in\SG$ in~\eqref{eq:pair_decomposition_Ad},
captures the vertical part of a pair $(q_0,q_1)$ in the sense that
``what is left'' is horizontal. The next definition makes this notion
more precise.

\begin{definition}
  \label{def:associated_discrete_connection_form}
  Given a discrete connection $\DC$ with domain $\jgU$ on the
  principal $\SG$-bundle $\pi:Q\rightarrow Q/\SG$, we define its
  \jdef{associated discrete connection form}
  \begin{equation}
    \label{eq:DC_associated_to_DC}
    \DC:\jgU\subset Q\times Q\rightarrow \SG \stext{ by }
    \DC(q_0,q_1):=g,
  \end{equation}
  where $g$ is the element of $\SG$ that appears in the
  decomposition~\eqref{eq:pair_decomposition_Ad}.
\end{definition}

When $\pi:Q\rightarrow Q/\SG$ is a principal $\SG$-bundle, the fibered
product of $\pi$ with itself ---that is, the pairs $(q_0,q_1)$ such
that $\pi(q_0) = \pi(q_1)$--- is denoted by $Q\times_{\pi\times \pi}
Q$. Let $\kappa:Q\times_{\pi\times \pi} Q\rightarrow \SG$ be defined
by $\kappa(q_0,q_1):=g$ if and only if $l^Q_g(q_0)=q_1$. It is easy
to check that $\kappa$ is a smooth function.

\begin{lemma}\label{le:CD_is_smooth}
  Given a discrete connection $\DC$ with domain $\jgU$ on the
  principal $\SG$-bundle $\pi:Q\rightarrow Q/\SG$, its associated
  discrete connection form $\DC:\jgU\rightarrow \SG$ is smooth.
\end{lemma}
\begin{proof}
  The function $\DC$ is a composition of smooth functions. Indeed, for
  all $(q_0,q_1)\in\jgU$, we have $\DC(q_0,q_1) =
  \kappa\big(p_2(((id_Q\times \pi)|_{Hor_\DC})^{-1}((id_Q\times
  \pi)(q_0,q_1))),q_1\big)$.
\end{proof}

\begin{example}\label{ex:trivial_bundle-discrete_connection-C}
  Let $R$ be a smooth connected manifold and $\SG$ a Lie group. Define
  $Q:=R\times \SG$ and consider the left $\SG$-action on $Q$ defined
  by $l^Q_g(r,g'):=(r,g g')$. This action turns $Q$ into the principal
  $\SG$-bundle $p_1:Q\rightarrow R$. Let $\mathcal{U}''\subset R\times
  R$ be an open subset containing the diagonal $\Delta_R$ and
  $C:\mathcal{U}''\rightarrow \SG$ be a smooth function such that
  $C(r_0,r_0) = e$ for all $r_0\in R$. Consider
  \begin{equation}\label{eq:trivial_bundle-discrete_connection-C-hor}
    \begin{split}
      Hor:=\{((r_0,g_0),(r_1,g_1))\in Q\times Q: & (r_0,r_1)\in
        \mathcal{U}''\text{ and } \\&  g_0 = g_1
        C(r_0,r_1)\}\subset Q\times Q.
    \end{split}
  \end{equation}
  $Hor\subset Q\times Q$ is a regular submanifold because $Hor$ is the
  graph of the smooth map $((r_0,g_0),r_1)\mapsto ((r_0,g_0),(r_1,g_0
  (C(r_0,r_1))^{-1}))$. It is immediate that $Hor$ is $l^{Q\times
    Q}$-invariant and contains $\Delta_Q$. Furthermore, as
  $(id_Q\times \pi)|_{Hor}:Hor \rightarrow Q\times (Q/\SG)$
  specializes to
  \begin{equation*}
    ((r_0,g_0),(r_1,g_0 (C(r_0,r_1))^{-1})) \mapsto
    ((r_0,g_0),r_1)
  \end{equation*}
  that is a diffeomorphism with inverse $((r_0,g_0),r_1) \mapsto
  ((r_0,g_0),(r_1,g_0 (C(r_0,r_1))^{-1}))$, we conclude that $Hor$
  defines a discrete connection $\DCp{C}$ on the (trivial) principal
  $\SG$-bundle $p_1:Q\rightarrow R$. It is easy to check that the
  domain of $\DCp{C}$ is $\jgU = (p_1\times p_1)^{-1}(\mathcal{U}'') =
  \{((r_0,g_0),(r_1,g_1)) \in Q\times Q : (r_0,r_1)\in
  \mathcal{U}''\}$, $\jgU' = \{ ((r_0,g_0),r_1) \in Q\times (Q/\SG) :
  (r_0,r_1)\in \mathcal{U}''\}$ and $\jgU'' = \mathcal{U}'' \subset
  R\times R$. In the special case when $\mathcal{U}''=R\times R$ and
  $C(r_0,r_1)=e$ for all $r_0,r_1\in R$, the connection $\DCp{e}$ is
  called the \jdef{trivial discrete connection}.

  Given $(q_0,q_1) = ((r_0,g_0),(r_1,g_1))\in \jgU$, it is easy to see
  that $(q_0,l^Q_{g^{-1}}(q_1)) \in Hor_{\DCp{C}}$ if and only if $g=g_1
  C(r_0,r_1) g_0^{-1}$. Therefore, the associated discrete connection
  form is
  \begin{equation}\label{eq:trivial_bundle-discrete_connection-C-1-form}
    \DCp{C}((r_0,g_0),(r_1,g_1)) = g_1 C(r_0,r_1) g_0^{-1}.
  \end{equation}

\end{example}

\begin{theorem}\label{thm:discrete_connection_vs_one_form}
  Let $\DC$ be a discrete connection on the principal $\SG$-bundle
  $\pi:Q\rightarrow Q/\SG$ with domain $\jgU$. Then, for all
  $(q_0,q_1)\in \jgU$ and $g_0,g_1\in \SG$,
  \begin{equation}
    \label{eq:Ad_GxG}
    \DC(l^Q_{g_0}(q_0),l^Q_{g_1}(q_1)) = g_1 \DC(q_0,q_1) g_0^{-1}.
  \end{equation}
  In addition, $Hor_\DC =\{(q_0,q_1)\in \jgU:
  \DC(q_0,q_1)=e\}$. Conversely, given a smooth function $\mathcal{A}:
  \mathcal{U}\rightarrow \SG$, where $\mathcal{U}\subset Q\times Q$ is
  an open subset that contains the diagonal $\Delta_Q\subset Q\times
  Q$ and is invariant under the product $\SG\times\SG$-action on
  $Q\times Q$, such that~\eqref{eq:Ad_GxG} holds (with $\DC$ replaced
  by $\mathcal{A}$) and $\mathcal{A}(q_0,q_0)=e$ for all $q_0\in Q$,
  then $Hor:=\{(q_0,q_1)\in \mathcal{U}: \mathcal{A}(q_0,q_1)=e\}$
  defines a discrete connection whose associated discrete connection
  form is $\mathcal{A}$.
\end{theorem}

\begin{proof}
  For $(q_0,q_1)\in\jgU$, let $h:=\DC(q_0,q_1)$. By the
  $\SG\times\SG$-invariance of $\jgU$, we have that
  $(l^Q_{g_0}(q_0),l^Q_{g_1}(q_1)) \in \jgU$ for any $g_0,g_1\in\SG$;
  let $\ti{h}:=\DC(l^Q_{g_0}(q_0),l^Q_{g_1}(q_1))$. By definition,
  $(q_0,l^Q_{h^{-1}}(q_1))\in Hor_\DC$ and by the $\SG$-invariance of
  $Hor_\DC$, we see that $(l^Q_{g_0}(q_0), l^Q_{(g_1 h
    g_0^{-1})^{-1}}(l^Q_{g_1}(q_1)))\in Hor_\DC$. On the other hand,
  also by definition, $(l^Q_{g_0}(q_0),
  l^Q_{\ti{h}^{-1}}(l^Q_{g_1}(q_1))) \in Hor_\DC$. Using that
  $(id_Q\times \pi)|_{Hor_\DC}$ is one to one, we conclude that $g_1 h
  g_0^{-1} = \ti{h}$, proving that~\eqref{eq:Ad_GxG} holds.

  By Proposition~\ref{prop:CD_imp_pair_decomposition}, the element $g$
  appearing in~\eqref{eq:pair_decomposition_Ad} is unique, hence $g=e$
  characterizes the horizontality of $(q_0,q_1)\in \jgU$. This proves
  that $Hor_\DC =\{(q_0,q_1)\in \jgU: \DC(q_0,q_1)=e\}$.

  Conversely, given $\mathcal{U}$, $\mathcal{A}$ and $Hor$ as in the
  statement, we show that $Hor$ defines a discrete connection.  Since
  $\mathcal{A}(q_0,q_0)=e$ for all $q_0\in Q$, we have that
  $\Delta_Q\subset Hor$. It is easy to check explicitly that
  $d\mathcal{A}:T_{(q_0,q_1)}\mathcal{U} \rightarrow T_e\SG =\jgsg$ is
  onto for all $(q_0,q_1)\in Hor$. Hence, $e$ is a regular value of
  $\mathcal{A}$ and $Hor :=\mathcal{A}^{-1}(\{e\})\subset Q\times Q$
  is a submanifold. The $l^{Q\times Q}$-invariance of $Hor$ follows
  readily using~\eqref{eq:Ad_GxG}.

  Assume now that $(q_0,l^Q_g(q_1))\in Hor$ with $(q_0,q_1)\in
  Hor$. Then $e= \mathcal{A}(q_0,l^q_g(q_1)) = g \mathcal{A}(q_0,q_1)
  = g e = g$, showing that
  condition~\ref{it:dconn_diffeo_vs_local_diffeo_and_intersection-inter}
  in Lemma~\ref{le:dconn_diffeo_vs_local_diffeo_and_intersection}
  holds. That $(id_Q\times \pi)|_{Hor}$ is a local diffeomorphism is
  checked locally. Let $U\subset Q/\SG$ be an open subset trivializing
  $\pi$ and $\sigma\in \Gamma(U,Q|_U)$, a local section over $U$ of
  the principal $\SG$-bundle $\pi:Q\rightarrow Q/\SG$. Define
  $\Phi_\sigma: Q\times U\rightarrow Q\times (Q|_U)$ by
  $\Phi_\sigma(q_0,r_1) :=
  (q_0,l^Q_{\mathcal{A}(q_0,\sigma(r_1))^{-1}}(\sigma(r_1)))$. It can
  be checked that $\Phi_\sigma$ is a smooth map whose image is
  contained in $Hor$. Furthermore, $\Phi_\sigma$ is the inverse of
  $(id_Q\times \pi)|_{Hor \cap ((Q|_U)\times (Q|_U))}$, showing that
  condition~\ref{it:dconn_diffeo_vs_local_diffeo_and_intersection-ldiff}
  in Lemma~\ref{le:dconn_diffeo_vs_local_diffeo_and_intersection}
  holds. By
  Lemma~\ref{le:dconn_diffeo_vs_local_diffeo_and_intersection},
  $(id_Q\times \pi)|_{Hor}$ is an injective local diffeomorphism and
  we conclude that $Hor$ defines a discrete connection $\DC$ on
  $\pi:Q\rightarrow Q/\SG$. Direct evaluation shows that the domain of
  $\DC$ is $\jgU = \mathcal{U}$ and that the discrete connection form
  associated to $\DC$ is $\mathcal{A}$.
\end{proof}

Motivated by the previous analysis we introduce the following concept.

\begin{definition}\label{def:discrete_connection_form}
  Let $\pi:Q\rightarrow Q/\SG$ be a principal $\SG$-bundle and
  $\mathcal{U}\subset Q\times Q$ be an open subset that contains the
  diagonal $\Delta_Q\subset Q\times Q$ and is invariant under the
  product $\SG\times\SG$-action on $Q\times Q$. A smooth function
  $\mathcal{A}:\mathcal{U}\rightarrow \SG$ is called a \jdef{discrete
    connection form} if $\mathcal{A}(q_0,q_0)=e$ for all $q_0\in Q$
  and it satisfies
  \begin{equation*}
    \mathcal{A}(l^Q_{g_0}(q_0),l^Q_{g_1}(q_1)) = g_1 \mathcal{A}(q_0,q_1) g_0^{-1}
    \stext{ for all } (q_0,q_1)\in \jgU,\quad g_0,g_1\in \SG.
  \end{equation*}
\end{definition}

Using the new notion and taking Lemma~\ref{le:CD_is_smooth} into account,
we rewrite Theorem~\ref{thm:discrete_connection_vs_one_form} as
follows.

\begin{theorem}\label{thm:discrete_connection_vs_one_form-equivalence}
  Let $\DC$ be a discrete connection on the principal $\SG$-bundle
  $\pi:Q\rightarrow Q/\SG$ with domain $\jgU$. Then, its associated
  discrete connection form $\DC$ is a discrete connection
  form. Conversely, given a discrete connection form $\mathcal{A}:
  \mathcal{U}\rightarrow \SG$ on the same principal bundle $\pi$, with
  $\mathcal{U}$ as in Definition~\ref{def:discrete_connection_form},
  the subset $Hor:=\{(q_0,q_1)\in \mathcal{U}:
  \mathcal{A}(q_0,q_1)=e\} \subset Q\times Q$ defines a discrete
  connection on $\pi$ whose associated discrete connection form is
  $\mathcal{A}$.
\end{theorem}

\begin{remark}
  Theorem~\ref{thm:discrete_connection_vs_one_form-equivalence}
  establishes a correspondence between discrete connections and
  discrete connection forms. The assignments $\mathcal{A}\mapsto
  Hor_\mathcal{A}:=\mathcal{A}^{-1}(\{e\})$ and $Hor_\DC \mapsto \DC$
  given by~\eqref{eq:DC_associated_to_DC} are the corresponding
  opposite operations. This correspondence justifies using the name
  $\DC$ for both the discrete connection and the discrete connection
  form.
\end{remark}

\begin{remark}
  The situation described in
  Example~\ref{ex:trivial_bundle-discrete_connection-C} corresponds to
  the general discrete connection on the (trivial) principal
  $\SG$-bundle $p_1:R\times \SG\rightarrow R$. Indeed, if $\DC$ is
  such a connection with domain $\jgU$, using~\eqref{eq:Ad_GxG}, we
  have that $\DC((r_0,g_0),(r_1,g_1)) = g_1 \DC((r_0,e),(r_1,e))
  g_0^{-1}$.
  % \begin{equation*}
  %   \DC((r_0,g_0),(r_1,g_1)) = g_1 \DC((r_0,e),(r_1,e)) g_0^{-1}.
  % \end{equation*}
  Define $C(r_0,r_1):= \DC((r_0,e),(r_1,e))$ for all $(r_0,r_1)\in
  \jgU'' = (\pi\times \pi)(\jgU)$, so that
  \begin{equation*}
    Hor_\DC = \DC^{-1}(\{e\}) =
    \{((r_0,g_0),(r_1,g_1))\in \jgU : g_1 C(r_0,r_1) = g_0\}.
  \end{equation*}
  Comparison of this last expression
  with~\eqref{eq:trivial_bundle-discrete_connection-C-hor} shows that
  the discrete connection form satisfies $\DC = \DCp{C}$.

  As all principal $\SG$-bundles are locally of the form $p_1:R\times
  \SG\rightarrow R$,
  Example~\ref{ex:trivial_bundle-discrete_connection-C} provides a
  local description of arbitrary discrete connections on principal
  $\SG$-bundles.
\end{remark}

%%%%%%%%%%%%%%%%%%%%%%%%%%%%%%%%%%%%%%%%%%%%%%%%%%%%%%%%%%%%%%%%%%%

\section{Discrete horizontal lift}
\label{sec:horizontal_lift}

As in the case of connections on principal bundles, discrete
connections establish local diffeomorphisms between slices of the
horizontal submanifold and the base space of the bundle. The inverse
operation is the \jdef{discrete horizontal lift}.

\begin{definition}
  Let $\DC$ be a discrete connection with domain $\jgU$ on the
  principal $\SG$-bundle $\pi:Q\rightarrow Q/\SG$ ; let
  $\jgU'\subset Q\times (Q/\SG)$ be the open set defined in
  part~\ref{it:dconn_basic_prop-open_sets} of
  Proposition~\ref{prop:dconn_basic_prop}. The \jdef{associated
    discrete horizontal lift} $\HLd{}:\jgU'\rightarrow Q\times Q$ is
  the inverse map of the diffeomorphism $(id_Q\times
  \pi)|_{Hor_\DC}:Hor_\DC\rightarrow \jgU'$. Explicitly
  \begin{equation}\label{eq:HLd-def}
    \HLd{}(q_0,r_1) \:= (q_0,q_1) \quad \iff \quad (q_0,q_1)\in
    Hor_{\DC} \ \text{ and }\  \pi(q_1)=r_1.
  \end{equation}
  It is convenient to write $\HLd{q_0}(r_1):= \HLd{}(q_0,r_1)$ and
  define $\HLds{q_0} :=p_2\circ
  \HLd{q_0}$.
\end{definition}

\begin{example}\label{ex:trivial_bundle-horizontal_lift-C}
  The horizontal lift associated to the connections $\DCp{C}$
  introduced in Example~\ref{ex:trivial_bundle-discrete_connection-C} is
  given by $\HLd{}:(R\times\SG)\times R \rightarrow (R\times
  \SG)\times (R\times \SG)$ with
  \begin{equation*}
    \HLd{(r_0,g_0)}(r_1) = ((r_0,g_0),(r_1,g_0 C(r_0,r_1)^{-1})).
  \end{equation*}
\end{example}

\begin{remark}
  When $(q_0,q_1)\in\jgU$, if $g:=\DC(q_0,q_1)$, we know that
  $(q_0,l^Q_{g^{-1}}(q_1)) \in  Hor_\DC$. Then, $\HLd{q_0}(\pi(q_1)) =
  \HLd{q_0}(\pi(l^Q_{g^{-1}}(q_1))) = (q_0,l^Q_{g^{-1}}(q_1))$ and
  $\HLds{q_0}(\pi(q_1)) = l^Q_{g^{-1}}(q_1)$, so that
  $l^Q_g(\HLds{q_0}(\pi(q_1))) = q_1$. Hence,
  \begin{equation}
    \label{eq:relation_between_CD_and_HL}
    \DC(q_0,q_1) = \kappa(\HLds{q_0}(\pi(q_1)),q_1), \stext{ for all } 
    (q_0,q_1)\in \jgU.
  \end{equation}
\end{remark}

The following result establishes the basic properties of the discrete
horizontal lift associated to a discrete connection. It also proves
that a discrete connection can be reconstructed given its associated
horizontal lift.

\begin{theorem}\label{thm:horizontal_lift_properties_and_characterization}
  Let $\DC$ be a discrete connection on the principal $\SG$-bundle
  $\pi:Q\rightarrow Q/\SG$ with domain $\jgU$. Then the following
  assertions are true.
  \begin{enumerate}
  \item \label{it:HL_properties-U'_is_invariant} $\jgU'\subset Q\times
    (Q/\SG)$ is $\SG$-invariant for the $\SG$-action defined by
    $l^{Q\times (Q/\SG)}_g(q_0,r_1) := (l^Q_g(q_0),r_1)$ for all
    $g\in\SG$.
  \item \label{it:HL_properties-HL_is_smooth} $\HLd{}:\jgU'\rightarrow
    Q\times Q$ is smooth and $\SG$-equivariant for the $\SG$-actions
    $l^{Q\times(Q/\SG)}$ and $l^{Q\times Q}$.
  \item \label{it:HL_properties-HL_is_a_section} $\HLd{}$ is a section
    over $\jgU'$ of $(id_Q\times \pi):Q\times Q\rightarrow Q\times
    (Q/\SG)$; that is, $(id_Q\times \pi) \circ \HLd{} = id_{\jgU'}$.
  \item \label{it:HL_properties-HL_normalization} For every $q_0\in
    Q$, $(q_0,\pi(q_0)) \in \jgU'$ and $\HLd{q_0}(\pi(q_0)) =
    (q_0,q_0)$.
  \end{enumerate}
  Conversely, assume that $\mathcal{U}'\subset Q\times (Q/\SG)$ is an
  open set that satisfies
  condition~\ref{it:HL_properties-U'_is_invariant} (with $\jgU'$
  replaced by $\mathcal{U}'$) and $h:\mathcal{U}'\rightarrow Q\times
  Q$ is a map such that
  conditions~\ref{it:HL_properties-HL_is_smooth},
  \ref{it:HL_properties-HL_is_a_section}
  and~\ref{it:HL_properties-HL_normalization} are satisfied (with
  $\jgU'$ and $\HLd{}$ replaced by $\mathcal{U}'$ and $h$). Then,
  there exists a unique discrete connection $\DC$ with domain $\jgU =
  (id_Q\times \pi)^{-1}(\mathcal{U}')$ on $\pi:Q\rightarrow Q/\SG$
  such that $\mathcal{U}'=\jgU'$ and $h=\HLd{}$.
\end{theorem}

\begin{proof}
  In general, $id_Q\times \pi$ is $\SG$-equivariant for the actions
  $l^{Q\times Q}$ and $l^{Q\times (Q/\SG)}$. When $\DC$ is a discrete
  connection, $Hor_{\DC}$ is $\SG$-invariant, so that $(id_Q\times
  \pi)|_{Hor_\DC}$ is a $\SG$-equivariant diffeomorphism and,
  consequently, its image $\jgU'$ is $\SG$-invariant and its inverse
  $\HLd{}$ is smooth and $\SG$-equivariant. This proves
  points~\ref{it:HL_properties-U'_is_invariant}
  and~\ref{it:HL_properties-HL_is_smooth}.
  Point~\ref{it:HL_properties-HL_is_a_section} follows immediately
  from the definition of $\HLd{}$, while
  point~\ref{it:HL_properties-HL_normalization} is a consequence of
  $\Delta_Q\subset Hor_\DC$.

  Now assume that $\mathcal{U}'$ and $h$ are as in the statement. Let
  $\jgU:=(id_Q\times \pi)^{-1}(\mathcal{U}')$, that is open due to the
  openness of $\mathcal{U}'$ and the smoothness of $id_Q\times
  \pi$. Furthermore, it follows from the $\SG$-invariance of
  $\mathcal{U}'$ that $\jgU$ is $\SG\times \SG$-invariant and, by the
  first part of condition~\ref{it:HL_properties-HL_normalization}, it
  contains the diagonal $\Delta_Q\subset Q\times Q$.  Motivated
  by~\eqref{eq:relation_between_CD_and_HL}, define
  $\DC:\jgU\rightarrow \SG$ by $\DC(q_0,q_1) :=
  \kappa(p_2(h(q_0,\pi(q_1))),q_1)$. Being a composition of smooth
  functions, $\DC$ is smooth. Straightforward computations show that
  $\DC$ satisfies condition~\eqref{eq:Ad_GxG} and $\DC(q_0,q_0)=e$ for
  all $q_0\in Q$. All together, by
  Theorem~\ref{thm:discrete_connection_vs_one_form}, $\DC$ defines a
  discrete connection with domain $\jgU$ on $\pi:Q\rightarrow
  Q/\SG$. It is immediate that $\jgU' = (id_Q\times \pi)(\jgU) =
  \mathcal{U}'$ and, using~\eqref{eq:relation_between_CD_and_HL} as
  well as the definition of $\DC$, we conclude that $h = \HLd{}$,
  proving the last part of the Theorem. The uniqueness assertion
  follows from the fact that the discrete connection form is
  determined by the horizontal lift
  using~\eqref{eq:relation_between_CD_and_HL}.
\end{proof}

Motivated by the previous analysis, it is convenient to introduce the
following notion.

\begin{definition}\label{def:HLd}
  Let $\SG$ be a Lie group acting on $Q$ by $l^Q$ in such a way that
  $\pi:Q\rightarrow Q/\SG$ is a principal $\SG$-bundle and let
  $\mathcal{U}'\subset Q\times (Q/\SG)$ be an open subset. A smooth
  function $h:\mathcal{U}'\rightarrow Q\times Q$ is a \jdef{discrete
    horizontal lift} on $\pi$ if the following conditions hold.
  \begin{enumerate}
  \item \label{it:HL_properties-U'_is_invariant} $\mathcal{U}'\subset
    Q\times (Q/\SG)$ is $\SG$-invariant for the $\SG$-action
    $l^{Q\times (Q/\SG)}$.
  \item \label{it:HL_properties-HL_is_smooth} $h:\mathcal{U}'\rightarrow
    Q\times Q$ is smooth and $\SG$-equivariant for the $\SG$-actions
    $l^{Q\times(Q/\SG)}$ and $l^{Q\times Q}$.
  \item \label{it:HL_properties-HL_is_a_section} $h$ is a section of
    $(id_Q\times \pi):Q\times Q\rightarrow Q\times (Q/\SG)$ over
    $\mathcal{U}'$; that is, $(id_Q\times \pi) \circ h =
    id_{\mathcal{U}'}$.
  \item \label{it:HL_properties-HL_normalization} For every $q_0\in
    Q$, $(q_0,\pi(q_0)) \in \mathcal{U}'$ and $h(q_0,\pi(q_0)) =
    (q_0,q_0)$.
  \end{enumerate}
\end{definition}

We can rewrite
Theorem~\ref{thm:horizontal_lift_properties_and_characterization}
using the new concept as follows.

\begin{theorem}
  \label{thm:horizontal_lift_properties_and_characterization-rewrite}
  Let $\DC$ be a discrete connection on the principal $\SG$-bundle
  $\pi:Q\rightarrow Q/\SG$ with domain $\jgU$. Then
  $\HLd{}:\jgU'\rightarrow Q\times Q$ as defined by~\eqref{eq:HLd-def}
  is a discrete horizontal lift on $\pi$.  Conversely, if
  $h:\mathcal{U}'\rightarrow Q\times Q$ is a discrete horizontal lift
  on $\pi:Q\rightarrow Q/\SG$, there exists a unique discrete
  connection $\DC$ with domain $\jgU = (id_Q\times
  \pi)^{-1}(\mathcal{U}')$ on $\pi$ such that $\mathcal{U}'=\jgU'$ and
  $h=\HLd{}$.
\end{theorem}

\begin{remark}
  We see from
  Theorem~\ref{thm:horizontal_lift_properties_and_characterization-rewrite}
  that mapping $\DC\mapsto \HLd{}$ and $h\mapsto \mathcal{A}$ with
  $\mathcal{A}(q_0,q_1) := \kappa(p_2(h(q_0,\pi(q_1))),q_1)$
  establishes a bijection between discrete horizontal lifts and
  discrete connections forms ---or discrete connections--- on a
  principal bundle.
\end{remark}

\begin{remark}\label{rem:existence_nonexistence_of_discrete_connections}
  Given a discrete connection $\DC$ on the principal bundle
  $\pi:Q\rightarrow Q/\SG$, its discrete connection form $\DC$ and
  discrete horizontal lift $\HLd{}$ may not be defined
  everywhere. Indeed, if $\HLd{}:Q\times (Q/\SG)\rightarrow Q\times Q$
  (that is, $\jgU'=Q\times (Q/\SG)$), then for any $q\in Q$, the map
  $r\mapsto \HLds{q}(r)$ is a global section of the principal bundle
  $\pi$, so that the bundle is trivial. Hence, for nontrivial
  principal $\SG$-bundles $\HLd{}$ and, consequently, $\DC$ can only
  be defined in some open set of the total space.
\end{remark}

\begin{remark}
  The equations of motion of a $\SG$-symmetric mechanical system on
  $Q$ can be written using a connection $\CC$ on the principal
  $\SG$-bundle $\pi : Q \rightarrow Q / \SG$. In
  \cite{ar:cendra_marsden_ratiu-geometric_mechanics_lagrangian_reduction_and_nonholonomic_systems},
  Cendra, Marsden and Ratiu use the same type of connection to
  construct an isomorphic model for the reduced space $TQ/ \SG$. Given
  a connection $\CC$, they define $\alpha_\CC : TQ/ \SG \rightarrow
  T(Q/\SG) \oplus \ti{\jgsg}$, an isomorphism of vector bundles over
  $Q/G$, by
  \begin{equation*}
    \alpha_\CC([q,\dot{q}]_{\SG}) := \big(
    d\pi(q,\dot{q}), [(q, \CC(q,\dot{q}))]_\SG\big),
  \end{equation*}
  where $\ti{\jgsg}$ is the adjoint vector bundle of $\jgsg$.  This
  identification allows them to establish a reduced variational
  principle and the associated reduced equations of motion.

  Similarly, a discrete connection $\DC$ on the principal bundle
  $\pi:Q\rightarrow Q/\SG$ can be used to construct an isomorphic
  model for the discrete reduced space $(Q \times Q) / \SG$. Given
  $\DC$, we define $\alpha_{\DC} : (Q \times Q) / \SG \rightarrow
  (Q/\SG \times Q/\SG) \times_{Q/\SG} \tilde{\SG}$, an isomorphism of
  bundles over $Q/\SG$, by
  \begin{equation*}
    \alpha_\DC([(q_0,q_1)]_{\SG}) := ((\pi(q_0),\pi(q_1) ) , 
    [(q_0,\DC(q_0,q_1))]_{\SG}),
  \end{equation*}
  where $\tilde{\SG} := (Q\times \SG) / \SG$ is the adjoint bundle of
  $Q$ by $\SG$ with respect to the $\SG$-action on $Q \times \SG$
  given by $l_{g}^{Q \times \SG}(q,h) := (l_{g}^Q(q),ghg^{-1})$. This
  identification of spaces is used in
  \cite{ar:fernandez_tori_zuccalli-lagrangian_reduction_of_discrete_mechanical_systems},~\cite{ar:fernandez_tori_zuccalli-lagrangian_reduction_of_discrete_mechanical_systems_by_stages},
  and
  \cite{ar:marrero_martin_martinez-discrete_lagrangian_and_hamiltonian_mechanics_on_lie_groupoids}
  to study the reduction of discrete mechanical systems with
  symmetries. It is important to observe that, according to
  Remark~\ref{rem:existence_nonexistence_of_discrete_connections},
  this identification is only local for nontrivial $\SG$-bundles.
\end{remark}

%%%%%%%%%%%%%%%%%%%%%%%%%%%%%%%%%%%%%%%%%%%%%%%%%%%%%%%%%%%%%%%%%%%

\section{Existence of discrete connections}
\label{sec:existence_of_discrete_connections}

All principal $\SG$-bundles carry connections that can be constructed,
for instance, using Riemannian metrics on the bundle. In this section
we prove an existence result for discrete connections given
appropriate Riemannian metrics.

Let $\pi:Q\rightarrow Q/\SG$ be a principal $\SG$-bundle and
$\ip{}{}_Q$ a $\SG$-invariant Riemannian metric on $Q$. The vertical
bundle $\mathcal{V}$ has an orthogonal complement, the
\jdef{horizontal bundle} $\mathcal{H}\subset TQ$; it is easy to check
that $\mathcal{H}$ defines a connection $\CCp{\ip{}{}_Q}$ on the
principal $\SG$-bundle $\pi$. Then there is a unique metric
$\ip{}{}_{Q/\SG}$ on $Q/\SG$ that turns $\pi$ into a Riemannian
submersion, \ie, for each $q\in Q$,
$d\pi(q)|_{\mathcal{H}_q}:\mathcal{H}_q\rightarrow T_{\pi(q)}(Q/\SG)$
is an isometry.

By Theorems 8.7 in Chapter III and 3.6 in Chapter IV
in~\cite{bo:kobayashi_nomizu-foundations-v1}, for every $r\in Q/\SG$,
there is an open set $W_r\subset Q/\SG$ such that any two points in
$W_r$ can be joined by a unique length minimizing geodesic (with
respect to $\ip{}{}_{Q/\SG}$) lying in $W_r$. Furthermore, for each
$r'\in W_r$ there is a normal coordinate neighborhood centered at $r'$
containing $W_r$. Define 
\begin{equation}
  \label{eq:U_in_terms_of_W_r}
  \mathcal{U}:= \cup_{r\in Q/\SG} (\pi\times \pi)^{-1}(W_{r}) \subset Q\times Q,
\end{equation}
that is open because the sets $W_r$ are open and $\pi$ is
continuous. Also, from the definition, $\mathcal{U}$ is
$\SG\times\SG$-invariant (for the product $\SG$-action) and contains
the diagonal $\Delta_Q$.

Let $(q_0,q_1)\in \mathcal{U}$. As $\pi(q_0), \pi(q_1) \in W_r$ for
some $r\in Q/\SG$, they can be joined by a unique geodesic
$\gamma:[0,1]\rightarrow W_r$ such that $\gamma(0)=\pi(q_0)$ and
$\gamma(1)=\pi(q_1)$. Being $\CCp{\ip{}{}_Q}$ a connection on the
principal $\SG$-bundle $\pi: Q\rightarrow Q/\SG$, by Proposition 3.1
in~\cite{bo:kobayashi_nomizu-foundations-v1}, there is an
$\CCp{\ip{}{}_Q}$-horizontal lift\footnote{As we will be considering
  only one connection on $\pi$, in what follows, we omit its reference
  when we consider horizontal lifts of paths.}
$\ti{\gamma}:[0,1]\rightarrow Q|_{W_r}$ of $\gamma(t)$ to $Q$ such that
$\ti{\gamma}(0) = q_0$ and $\ti{\gamma}(1) \in Q_{\pi(q_1)}$. In
addition, as $\gamma$ is a geodesic and $\pi$ is a Riemannian
submersion, by Proposition 3.1
in~\cite{ar:hermann-a_sufficient_condition_that_a_mapping_of_riemannian_manifolds_be_a_fibre_bundle},
$\ti{\gamma}$ is a geodesic for $\ip{}{}_Q$. The value
$\ti{\gamma}(1)$ is independent of the open set $W_r$ chosen for the
construction. Indeed, if we pick another open set $W_{r'}$ containing
$\pi(q_0)$ and $\pi(q_1)$ we would have two length minimizing
geodesics $\gamma_r$ and $\gamma_{r'}$ joining $\pi(q_0)$ to
$\pi(q_1)$ and contained in $W_r$ and $W_{r'}$ respectively. Then, by
Theorem 10.4 in~\cite{bo:milnor-morse_theory}, $\gamma_r([0,1]) =
\gamma_{r'}([0,1])$, so that both geodesics are contained in $W_r\cap
W_{r'}$ and, by the uniqueness of the geodesics in $W_r$, $\gamma_r(t)
= \gamma_{r'}(t)$ for all $t\in [0,1]$.

\begin{remark}
  The open subsets $W_r\subset Q/\SG$ introduced above are not
  necessarily unique. The same, in principle, applies to the open subset
  $\mathcal{U}\subset Q\times Q$ defined by~\eqref{eq:U_in_terms_of_W_r}.
\end{remark}

Choosing a family $\{W_r:r\in Q/\SG\}$ as above and constructing
$\mathcal{U}$ with~\eqref{eq:U_in_terms_of_W_r}, define
$\DCp{\ip{}{}_Q}:\mathcal{U}\rightarrow \SG$ by
\begin{equation}
  \label{eq:DC_given_riemannian_metric}
  \DCp{\ip{}{}_Q}(q_0,q_1) := \kappa(\ti{\gamma}(1),q_1).
\end{equation}

\begin{theorem}\label{thm:existence_of_discrete_connections_given_riemm_metric}
  Let $(Q,\ip{}{}_Q)$ be a Riemannian manifold where the Lie group
  $\SG$ acts by isometries in such a way that $\pi:Q\rightarrow Q/\SG$
  is a principal $\SG$-bundle. Then, there is a discrete connection
  $\DCp{\ip{}{}_Q}$ on $\pi$ whose domain is $\jgU = \mathcal{U}$ and
  whose discrete connection form is given
  by~\eqref{eq:DC_given_riemannian_metric}.
\end{theorem}

\begin{proof}
  From the previous construction, $\mathcal{U}\subset Q\times Q$ is an
  open set, $\SG\times \SG$-invariant for the product $\SG$-action and
  contains the diagonal $\Delta_Q$. In addition, by the smooth
  dependence of the geodesics on both the initial and final point as
  well as the initial point and velocity, $\DCp{\ip{}{}_Q}$ is
  smooth. For any $q_0\in Q$, the unique length minimizing geodesic
  joining $\pi(q_0)$ to itself in $Q/\SG$ is the constant path, so
  that its horizontal lift is, again, the constant path and we
  conclude that
  \begin{equation*}
    \DCp{\ip{}{}_Q}(q_0,q_0) := \kappa(\ti{\gamma}(1),q_0) = 
    \kappa(q_0,q_0) = e.
  \end{equation*}
  Let $(q_0,q_1)\in \mathcal{U}$ and $g_0,g_1\in\SG$. for any $r\in
  Q/\SG$ such that $\pi(q_0), \pi(q_1)\in W_r$, there is a unique
  length minimizing geodesic $\gamma:[0,1]\rightarrow Q/\SG$ contained
  in $W_r$ and such that $\gamma(0)=\pi(q_0)$ and
  $\gamma(1)=\pi(q_1)$. As $\pi(q_0) = \pi(l^Q_{g_0}(q_0))$ we can
  consider the horizontal lifts $\ti{\gamma}_{q_0}$ and
  $\ti{\gamma}_{l^Q_{g_0}(q_0)}$ of $\gamma$ starting at $q_0$ and
  $l^Q_{g_0}(q_0)$ respectively. As $\SG$ acts on $Q$ by isometries,
  for any $g_0 \in \SG$, $l^Q_{g_0}$ is an isometry of $Q$. Then,
  $l^Q_{g_0}(\ti{\gamma}_{q_0}(t))$ is a horizontal geodesic in $Q$
  starting at $l^Q_{g_0}(q_0)$. Furthermore,
  $\pi(l^Q_{g_0}(\ti{\gamma}_{q_0}(t))) = \pi(\ti{\gamma}_{q_0}(t)) =
  \gamma(t)$, so that $l^Q_{g_0}(\ti{\gamma}_{q_0}(t))$ is the
  horizontal lift of $\gamma$ starting at $l^Q_{g_0}(q_0)$, that is,
  $\ti{\gamma}_{l^Q_{g_0}(q_0)} = l^Q_{g_0} \circ
  \ti{\gamma}_{q_0}$. Then
  \begin{equation*}
    \begin{split}
      \DCp{\ip{}{}_Q}(l^Q_{g_0}(q_0), l^Q_{g_1}(q_1)) =&
      \kappa(\ti{\gamma}_{l^Q_{g_0}(q_0)}(1), l^Q_{g_1}(q_1)) =
      \kappa(l^Q_{g_0}(\ti{\gamma}_{q_0}(1)), l^Q_{g_1}(q_1)) \\=& g_1
      \kappa(\ti{\gamma}_{q_0}(1), q_1) g_0^{-1} = g_1
      \DCp{\ip{}{}_p}(q_0,q_1) g_0^{-1},
    \end{split}
  \end{equation*}
  so that $\DCp{\ip{}{}_Q}$ satisfies condition~\eqref{eq:Ad_GxG}. The
  main result follows from
  Theorem~\ref{thm:discrete_connection_vs_one_form}.
\end{proof}

\begin{remark}\label{rem:horizontal_subamnifold_for_riemannian_metric}
  The submanifold $Hor_{\DCp{\ip{}{}_Q}}$ underlying the discrete
  connection $\DCp{\ip{}{}_Q}$ constructed in
  Theorem~\ref{thm:existence_of_discrete_connections_given_riemm_metric}
  is $Hor_{\DCp{\ip{}{}_Q}} = (\tau_Q\times \exp)(\mathcal{D}\cap
  \mathcal{H})$ where $\tau_Q:TQ\rightarrow Q$ is the canonical
  projection, $\mathcal{D}\subset TQ$ is an open subset ---containing
  the image of the zero section and contained in the domain of the
  exponential mapping--- and $\mathcal{H}$ is the horizontal
  distribution.
\end{remark}

\begin{example}\label{ex:riemannian_DC_in_hopf_bundle}
  Let $\HH$ be the $\R$-algebra of \jdef{quaternions} together with
  its canonical inner product $\ip{}{}_\HH$ and basis $\{\Qo, \Qi,
  \Qj, \Qk\}$. The submanifold of unit norm quaternions is the sphere
  $S^3$ with the round metric, while the unit norm imaginary
  quaternions form $S^2$ with the round metric. Define $\phi:S^3
  \rightarrow S^2$ by $\phi(q) := \jconj{q} \Qi q$, where $\jconj{q}$
  denotes the conjugated quaternion (change the sign of the imaginary
  part of $q$). It is a well known fact that $\phi$ is a principal
  $U(1)$-bundle for the $U(1)$-action on $S^3$ given by the isometries
  $l^{S^3}_{e^{i\theta}}(q) := (\cos(\theta)\Qo + \sin(\theta)\Qi)
  q$. This bundle is the \jdef{Hopf bundle}. The map $\phi$ is a
  Riemannian submersion if we consider \emph{twice} the round metric
  in $S^2$.

  The construction of a discrete connection $\DCp{\ip{}{}_{S^3}}$
  associated to the Hopf bundle described at the beginning of the
  section is as follows. When $q\in\HH-\{0\}$ we denote by $\langle q
  \rangle\subset \HH$ the $1$-dimensional \emph{real} subspace of
  $\HH$ generated by $q$.  For $q\in S^3$, the vertical bundle is
  $\mathcal{V}(q) = \langle \Qi q\rangle \subset T_qS^3 = \langle
  q\rangle ^\perp \subset \HH$, so that the horizontal bundle is
  $\mathcal{H}_{q} = \langle \Qi q\rangle^\perp \subset T_qS^3 \subset
  \HH$. For $r\in S^2$, the open subsets $W_r := \{ r' \in S^2 :
  \ip{r}{r'}_\HH >0\}\subset S^2$ have the required geodesic
  uniqueness properties. Then, $\mathcal{U} := \cup_{r\in S^2}
  (\phi\times \phi)^{-1}(W_r) = \{(q_0,q_1) \in S^3\times S^3 :
  \phi(q_0) \neq -\phi(q_1)\}$; the last identity is due to the fact
  that if $\phi(q_0) \neq -\phi(q_1)$ then $\phi(q_0),\phi(q_1)\in
  W_r$ for $r:=
  \frac{\phi(q_0)+\phi(q_1)}{\norm{\phi(q_0)+\phi(q_1)}}$. Using the
  explicit form of the geodesics in $S^2$ and $S^3$ as maximal
  circles, after some work, it can be seen that the domain of the
  discrete connection form $\DCp{\ip{}{}_{S^3}}$ is
  \begin{equation*}
    \jgU := \mathcal{U} = \{(q_0,q_1)\in
    S^3\times S^3: P_\Qo(q_1\conj{q_0})^2 + P_\Qi(q_1\conj{q_0})^2\neq 0\},
  \end{equation*}
  and the connection form is
  \begin{equation*}
    \DCp{\ip{}{}_{S^3}}(q_0,q_1) = \frac{P_\Qo(q_1\conj{q_0}) +  
      P_\Qi(q_1\conj{q_0}) i}{\abs{P_\Qo(q_1\conj{q_0}) + 
        P_\Qi(q_1\conj{q_0}) i}} \in U(1),
  \end{equation*}
  where $P_\Qo, P_\Qi : \HH\rightarrow \R$ are the projections on the
  corresponding coordinates.
\end{example}

\begin{remark}
  In the same general setting as in
  Theorem~\ref{thm:existence_of_discrete_connections_given_riemm_metric},
  Example 4.1
  of~\cite{ar:leok_marsden_weinstein-a_discrete_theory_of_connections_on_principal_bundles}
  constructs a function $\DC:Q\times Q\rightarrow \SG$ as
  follows. Given two points $q_0,q_1\in Q$, let $q_{01}$ be the
  geodesic in $Q$ satisfying $q_{01}(0)=q_0$ and $q_{01}(1)=1$ (this
  may actually restrict the domain of $\DC$ to an open neighborhood of
  the diagonal in $Q\times Q$). Then, let $x_{01} := \pi\circ q_{01}$
  and $\ti{q}_{01}$ the horizontal lift (with respect to the
  horizontal distribution $\mathcal{H}$) of $x_{01}$ to $Q$ starting
  at $q_0$. Finally, define
  \begin{equation}
    \label{eq:DC_ex_4.1_LMW}
    \DC(q_0,q_1) := \kappa(\ti{q}_{01}(1),q_1).
  \end{equation}
  Pairs in the ``horizontal manifold'' $\DC^{-1}(\{e\})$ are the
  endpoints of horizontal geodesics in $Q$. It follows that this
  manifold is essentially the same as $Hor_{\DCp{\ip{}{}_{Q}}}$
  corresponding to the discrete connection constructed by
  Theorem~\ref{thm:existence_of_discrete_connections_given_riemm_metric}
  (see Remark~\ref{rem:horizontal_subamnifold_for_riemannian_metric}).
  Still, the function $\DC$ defined in~\eqref{eq:DC_ex_4.1_LMW} may
  not be a discrete connection form: it may fail to satisfy
  condition~\eqref{eq:Ad_GxG}. For instance, in the context of the
  Hopf bundle considered in
  Example~\ref{ex:riemannian_DC_in_hopf_bundle}, it can be checked
  explicitly that, when $\theta\in (-\frac{\pi}{4},\frac{\pi}{4})$,
  \begin{equation*}
    \DC\left(\Qo, 
      l^{S^3}_{e^{i\theta}}\left(\frac{\Qo+\Qj}{\sqrt{2}}\right)\right) = 
    e^{i\beta_\theta} \stext{ for } 
    \beta_\theta := \frac{\sin(\theta)
      \arccos(\frac{\cos(\theta)}{\sqrt{2}})}{\sqrt{2-\cos(\theta)^2}},
  \end{equation*}
  whereas
  \begin{equation*}
    e^{i\theta} \DC\left(\Qo,\frac{\Qo+\Qj}{\sqrt{2}}\right) = e^{i\theta}.
  \end{equation*}
  Evaluation of the first derivatives of those expressions at
  $\theta=0$ confirms that~\eqref{eq:Ad_GxG} does not hold in this
  case.
\end{remark}

%%%%%%%%%%%%%%%%%%%%%%%%%%%%%%%%%%%%%%%%%%%%%%%%%%%%%%%%%%%%%%%%%%%

%\bibliography{math}   % name your BibTeX data base

%%%%%%%%%%%%%%%%%%%%%%%%%%%%%%%%%%%%%%%%%%%%%%%%%%%%%%%%%%%%%%%%%%%

\def\cprime{$'$} \def\polhk#1{\setbox0=\hbox{#1}{\ooalign{\hidewidth
  \lower1.5ex\hbox{`}\hidewidth\crcr\unhbox0}}} \def\cprime{$'$}
  \def\cprime{$'$}
\providecommand{\bysame}{\leavevmode\hbox to3em{\hrulefill}\thinspace}
\providecommand{\MR}{\relax\ifhmode\unskip\space\fi MR }
\providecommand{\MRhref}[2]{%
  \href{http://www.ams.org/mathscinet-getitem?mr=#1}{#2}
}
\providecommand{\href}[2]{#2}

%%%%%%%%%%%%%%%%%%%%%%%%%%%%%%%%%%%%%%%%%%%%%%%%%%%%%%%%%%%%%%%%

\end{document}